\documentclass[11pt]{amsart}

\usepackage{amsthm}
\usepackage{amsmath}
\usepackage{amssymb}
\usepackage{color}

\newtheorem{thm}{Theorem}[section]
\newtheorem{theorem}[thm]{Theorem}

\newtheorem{corollary}[thm]{Corollary}
\newtheorem{problem}[thm]{Problem}

\newtheorem{lemma}[thm]{Lemma}

\newtheorem{definition}[thm]{Definition}
\theoremstyle{remark}

\newcommand{\RR}{\mathbb R}

\newcommand{\CC}{\mathbb C}

\newcommand{\spn}{\textup{span}}

\newcommand{\pc}[1]{{\color{magenta}{#1}}}
\newcommand{\dg}[1]{{\color{blue}{#1}}}

\begin{document}

\title{Phase retrieval by projections in $\RR^n$ requires ${\large \bf{2n-2}}$ projections\\
(This paper has a gap)}
\author[Casazza, Ghoreishi
 ]{Peter G. Casazza and Dorsa Ghoreishi}
\address{Department of Mathematics, University
of Missouri, Columbia, MO 65211-4100}
\address{}

\thanks{The authors were supported by
 NSF DMS 1609760 and 1906725}

\email{Casazzap@missouri.edu}
\email{dorsa.ghoreishi@slu.edu}


\begin{abstract}
We will answer the most significant open problem in real phase retrieval by projections by showing it requires at least
$2n-2$ projections to do phase retrieval in $\RR^n$.
\end{abstract}

\maketitle

\section{Introduction}

Phase retrieval is one of the most active and applied subjects in mathematics today, with applications to pure mathematics, applied mathematics, engineering, medicine, computer science, and more.  Frame theory is fundamental to the digitalizing of information for processing and storing.  Phase retrieval has applications to X-ray crystallography, electron microscopy, astronomical imaging, optics, quantum physics, and much more.  Phase retrieval will even be needed to align the mirrors in the new James Webb Space Telescope. 

\vspace{0.2 in}

\begin{definition}
A family of vectors $\{\phi_i\}_{i=1}^m$ in $\RR^n$ or $\CC^n$ does {\bf phase retrieval} if for any non-zero $x,y\in \RR^n$ satisfying $|\langle x,\phi_i\rangle|
= |\langle y,\phi_i\rangle|$ for all $i\in [m]$ this will imply $x=cy$ with $|c|=1$.
\end{definition}
A fundamental idea here is:

\begin{definition} A family of vectors $\{\phi_i\}_{i=1}^m$ in $\RR^n$ has the {\bf complement property} if for all subsets $I\subset [m]$,
either \, $\spn \{\phi_i\}_{i\in I}=\RR^n$ or \ $\spn\{\phi_i\}_{i\in I^c}=\RR^n$.
\end{definition}

The most important classification of phase retrieval in $\RR^n$ appeared in \cite{BCE}.

\begin{theorem}
A family of vectors $\{\phi_i\}_{i=1}^m$ does phase retrieval in $\RR^n$ if and only if it has the complement property.
\end{theorem}

\begin{corollary}
If a family of vectors $\{\phi_i\}_{i=1}^m$ satisfies phase retrieval in $\RR^n$, then $m\ge 2n-1$.
\end{corollary}

In some applications a signal must be reconstructed from the norms of higher dimensional components. In X-ray crystallography
such a problem arises in crystal twinning \cite{Cr}. In this scenario, there exists a similar phase retrieval problem.

\begin{definition}
A family of subspaces $\{W_i\}_{i=1}^m$ in $\RR^n$or $\CC^n$ with respective orthogonal projections $\{P_i\}_{i=1}^m$ satisfy {\bf phase retrieval}
if for all $x,y\in \RR^n$, $\|P_ix\|=\|P_iy\|$ for all $i\in [m]$, implies $x=cy$ with $|c|=1$.
\end{definition}

A deep study of phase retrieval by projections appears in \cite{CCJ}. This includes the result:

\begin{theorem}
For any $0<r_i<n$, $i \in [2n-1]$ there are subspaces $\displaystyle \{W_i\}_{i=1}^{2n-1}$ in $\RR^n$ doing phase retrieval with {\textup{dim}} $W_i=r_i$.
\end{theorem}

A fundamental result in this area from \cite{CCJ} is:

\begin{theorem}
Let $\{W_i\}_{i=1}^m$ be subspaces of $\RR^n$. The following are equivalent:
\begin{enumerate}
\item $\{W_i\}_{i=1}^m$ does phase retrieval.
\item Whenever we choose orthonormal bases $\{u_{ij}\}_{j=1}^{I_i}$ for $W_i$, the family of  vectors$\{u_{ij}\}_{i=1,j=1}^{ \ m\ \ I_i}$ does
phase retrieval.
\end{enumerate}
\end{theorem}

\vspace{0.1 in}
Edidin \cite{E} proved a fundamental
result concerning phase retrieval by projections (See \cite{C} for an elementary proof of this resut.)

\vspace{0.1 in}

\begin{theorem}\label{edidin}
A family of projections $\{P_i\}_{i=1}^m$ does phase retrieval in $\RR^n$ if and only if for every $0\not= x\in \RR^n$,
$\spn\{P_ix\}_{i=1}^m=\RR^n$.
\end{theorem}

\vspace{0.1 in}
The following theorem in \cite{X} shows a particular case of $\RR^4$ in which phase retrieval is possible with $2n-2= 2(4)-2=6$ subspaces.

\vspace{0.1 in}
\begin{theorem}
There are 6 two-dimensional subspaces of $\RR^4$ satisfying phase retrieval.
\end{theorem}

In \cite{A} it was shown that there are 6 hyperplanes in $\RR^4$ doing phase retrieval. 

The following result also appears in
\cite{A}. Since this result is central to this paper, we include its short proof.

\begin{theorem}\label{T}
It takes at least $2n-2$ hyperplanes in $\RR^n$ to do phase retrieval.
\end{theorem}

\begin{proof}
Let $\{W_i\}_{i=1}^m$ be a family of hyperplanes in $\RR^n$ doing phase retrieval with respective projections $\{P_i\}_{i=1}^m$. \\
By way of contradiction, assume $m \leq 2n-3$.
Choose  a vector $$0 \not= x \in \bigcap_{i=1}^{n-1}W_i$$ so $P_ix=x$ for $i \in [n-1]$. Therefore we can write, $\{P_i x\}_{i=1}^{m}= \{x\} \cup\{P_ix\}_{i=n}^{m}$. Notice that $2n-3-(n-1)+1= n-1$ which shows $\{P_i x\}_{i=1}^m$ has at most $n-1$ non-zero vectors and cannot span the space. This is contradicting Theorem \ref{edidin}.

\end{proof}

The main open question in this area has been about finding the least number of subspaces satisfying phase retrieval. We will
show in this paper that it takes at least $2n-2$ subspaces to satisfy phase retrieval in $\RR^n$.  \\

Another open problem in this area is motivated by this theorem of 
Edidin \cite{E}:
\vspace{0.1 in}

\begin{theorem}
If $n=2^k-1$, for some $k$, then it takes $2n-1$ subspaces of $\RR^n$ to do phase retrieval.
\end{theorem}
\vspace{0.1 in}

So now the main open problem in this area is:
\vspace{0.1 in}

\begin{problem}
Can we do phase retrieval in $\RR^n$ with $2n-2$ subspaces whenever $n\not= 2^k-1$ for any $k$.
\end{problem}

\vspace{0.1 in}
The complex case is significantly more complicated. It is known \cite{BCMN} that phase retrieval can be done in $\CC^n$ with
$4n-4$ vectors. It was believed this result was best possible. But Vinzant \cite{V} showed that phase retrieval can be
done in $\CC^4$ with 11 vectors (in other words $4n-5$ vectors). So the major open problem here is:

\vspace{0.1 in}
\begin{problem}
What is the minimum number of vectors required to do phase retrieval in $\CC^n$?
\end{problem}
\vspace{0.1 in}

Phase retrieval by projections in $\CC^n$ is still unknown.

\vspace{0.1 in}

\begin{problem}
What is the minimum number of projections required to do phase retrieval in $\CC^n$
\end{problem}

\vspace{0.1 in}
Notice that Theorem \ref{edidin}
(Edidin's theorem) fails in $\CC^n$. A complex version of this theorem appears in \cite{X}, Theorem 2.1 (3). 
A completely different approach to a complex version of Edidin's theorem was proved in \cite{CC}. In this paper they associate vectors
in $\CC^n$  with rank 2 projections in $\RR^{2n}$.

\vspace{0.1 in}
\section{Main Result}
\vspace{0.1 in}

The following is immediate by Theorem \ref{edidin}.

\begin{lemma}
If a family of projections 
$\{P_i\}_{i=1}^m$ does phase retrieval in $\RR^n$, then for every $0\not= x\in \RR^n$, either $\spn \{P_ix\}_{i=2}^m=\RR^n$
or $span \{P_ix\}_{i=2}^m$ spans a hyperplane. \\
Moreover, if this is a hyperplane, $P_1x\notin \spn\{P_ix\}_{i=2}^m$.
\end{lemma}

\vspace{0.1 in}

The next result is fundamental to our proof.

\vspace{0.1 in}
\begin{theorem}\label{thm}
Let a family of projections $\{P_i\}_{i=1}^m$ do phase retrieval on $\RR^n$ and let
\[ F=\{x:\|x\|=1\mbox{ and }H_x=\spn\{P_ix\}_{i=2}^m \not= \RR^n\}.\]
If $x_1,x_2\in F$ and $H_{x_1}=H_{x_2}$, then $x_1=\pm x_2$.
\end{theorem}

\begin{proof}
Since $H_{x_1}=H_{x_2}$ is a hyperplane, there is a vector $\|z\|=1$ so that $z\perp P_ix_1$ and $z\perp P_ix_2$ for all
$i=2,3,\ldots,m$. Hence
\[ 0 = \langle z,P_ix_1\rangle =\langle P_iz,x_1\rangle=\langle z,P_ix_2\rangle=\langle P_iz,x_2\rangle, \mbox{ for all }
i=2,3,\ldots,m.\]
Since this family spans a hyperplane, $x_1,x_2$ are both orthogonal
to this hyperplane and so $x_1=\pm x_2$.
\end{proof}

\vspace{0.1 in}

The main result will be based on the following theorem:

\vspace{0.1 in} 

\begin{theorem}\label{TT}
Let a family of subspaces $\{W_i\}_{i=1}^m$ do phase retrieval in $\RR^n$. If $dim \ W_1^{\perp}\ge 2$, there is a $z\in W_1^{\perp}$ with $\|z\|=1$ so that
$\{W_1\cup z\} \cup \{W_2,\ldots,W_m\}$ does phase retrieval.
\end{theorem}

\begin{proof}
We will do this in steps.

\vskip10pt
(1) Choose orthonormal bases $\{u_{ij}\}_{j=1}^{I_i}$ for $W_i$  when $i \in [m]$
so that these vectors satisfy phase retrieval.  For $z\in W_1^{\perp}$ let $W_z=\spn\{W_1,z\}$
and let $P_z$ be the projection onto $W_z$.\\
We need to show there exists a $z\in W_1^{\perp}$ so that $\spn \{P_z,P_2,\ldots,P_n\}$ does phase retrieval in $\RR^n$.
\vskip10pt
(2) Choose $x\in \RR^n$ with $\|x\|=1$.  If $\spn \{P_ix\}_{i=2}^m=\RR^n$, this shows that if we add a vector $z$ which is in $W_1^{\perp}$, the subspaces still satisfy phase retrieval and we are done for this $x$.\\

Now we need to consider the cases when $\spn \{P_ix\}_{i=2}^m \neq \RR^n$,
\vskip10pt
(3)  By Theorem \ref{thm}, if \ $ \spn\{P_ix\}_{i=2}^m$ = $\spn \{P_iy\}_{i=2}^m\not= \RR^n$ and $\|x\|=\|y\|$, then $x=\pm y$.
\vskip10pt
(4) Note
\[ \spn\{P_ix\}_{i=2}^m=\spn\{u_{kj}:\langle P_ix,u_{kj}\rangle \not= 0\mbox{ for some }i=1,2,\ldots,m\}.\]

\vspace{0.1 in}

It follows from (3) that there are at most a finite number of $x's$ for which $\spn \{P_ix\}_{i=2}^m\not= \RR^n$.Call these $\{x_j\}_{j=1}^k$,
and let $V_{x_j}=\spn\{P_ix_j\}_{i=2}^m$ and choose $\|y_j\|=1$ with $y_j\perp V_{x_j}$.
\vskip10pt
(5)Now we need to find a $z\in W_1^{\perp}$ with $\|z\|=1$ and $P_zx_j\notin V_{x_j}$ for all $j \in [k]$. \\
Since there are a finite number of $V_{x_j}$, it suffices to show that for all $j\in [k]$ if  
\begin{equation}\label{E1}
A=\{z\in W_1^{\perp}:\|z\|=1,\  P_zx_j\in V_{x_j}=y_j^{\perp},\}
\end{equation}
the family of lines in the
unit ball of A going through the origin
is a set of measure zero relative to Lebesgue measure on $W_1^{\perp}$.
\vskip12pt
(6)  We have $A\cap y_j^{\perp}= A\cap x_j^{\perp} = \emptyset$.
\begin{itemize}
\item If $z\in A\cap x_j^{\perp}$ then:
\begin{align*}
P_zx_j &= P_1x_j+\langle x_j, z\rangle z \\
&=P_1x_j\in y_j^{\perp}.
\end{align*}

\vspace{0.05in}

It follows that  $\spn\{P_ix_j\}_{i=1}^m=y_j^{\perp}\not= \RR^n$. So by Theorem \ref{edidin}, $\{P_i\}_{i=1}^m$ fails
phase retrieval - a contradiction. 
\vskip10pt
\item If $z\in A\cap y_j^{\perp}$ then
\[ P_zx_j = P_1x_j+\langle x_j,z\rangle z\in y_j^{\perp}.\]

\vspace{0.05in}

Since $y_j^{\perp}$ is a subspace and $z\in y_j^{\perp}$ it follows again that $P_1x_j\in y_j \dg{^{\perp}}$ and so $\{P_i\}_{i=1}^m$
fails phase retrieval - again a contradiction.
\end{itemize}

\vskip10pt
(7) {\bf Claim:}  There is a constant $0< c\le 1$ so that for every $z\in A$ we have
\[ \langle x_j,z\rangle\langle y_j,z\rangle =c.\]
{\bf Proof of the claim:} Let $z_1,z_2\in A$. Then if $P_{z_1}x_j,P_{z_2}x_j\in y_j^{\perp}$, since $y_j^{\perp}$ is a subspace,
\[
P_{z_1}x_j-P_{z_2}x_j = (P_1x_j+\langle x_j,z_1\rangle z_1)-
(P_1x_j+\langle x_j,z_2\rangle z_2)\in y_j^{\perp}.\]
Hence
\[  \langle x_j,z_1\rangle \langle y_j,z_1\rangle -
 \langle x_j,z_2\rangle \langle y_j,z_2\rangle=0.\]
 
 \vspace{0.05in}
 
Notice that $c\not= 0$, since if $c=0$, then for every $z\in A$ either $z\perp x_j$ or $z\perp y_j$
which is not possible by (6).

\vspace{0.15in}

(8) After a rotation, we may assume $X=\spn \{x_j,y_j\} =\RR^2\subset \RR^n$ with canonical basis $\{e_1,e_2\}$ and
$x_j=e_1$.  \\
Let $y_j=(e,f)\in \RR^2$ and let $P$ be the projection onto $X$.
We want to find the points $v\in X$ so that $\langle x_j,v\rangle \langle y_j,v\rangle =c$.

\vspace{0.1in}

The points on the line $\ell_1$ orthogonal to $x_j=e_1$ and passing through $(x,0)$ represents all the points in $\RR^2$
whose inner product with $x_j$ is $x$. Let $m$ be the slope of the line orthogonal to $y_j$. The points on the line
$\ell_2$ orthogonal to $y_j$ and passing through $\frac{c}{x}(e,f)$ are the points in $\RR^2$ whose inner product with $y_j$
is $\frac{c}{x}$ since $e^2+f^2=1$. This line has equation:
\[ y-\frac{c}{x}f=m(x-\frac{c}{x}e).\]
That is:
\[ y=mx+\frac{c}{x}(f-em).\]
So if $v\in \ell_1\cap \ell_2$ for some $c\le x\le 1$ then
\[ v=(x,mx+\frac{c}{x}(f-em)).\]

\vspace{0.05in}
Therefore, 
the points $v\in X$ so that $\langle x_j,v\rangle \langle y_j,v\rangle =c$ are:\[ \{(x,mx +\frac{c}{x}(f-em)):c\le x \le 1\} \hspace{0.2in} (*)\]
\vskip10pt
(9)  We will answer the case where $dim\ W_1^{\perp}=2$.
We will examine 3 cases.
\vskip10pt

\noindent {\bf Case 1:} P maps $W_1^{\perp}$ onto $X$.
\vskip10pt
Since $P$ is onto, there is a $v\in W_1^{\perp}$ with $\|v\|=1$ and $Pv=re_1$.  Let $v^{\perp}$ be the unit vector in $W_1^{\perp}$ orthogonal
to $v$ and $y_j=(e,f)\in \RR^2$. Let $Pv^{\perp}=(g,h)$.
If $z\in A$, $z= av^{\perp}+bv$ and $a^2+b^2=1$.
By (8) and the argument above, $Pz=(x,mx+\frac{c}{x}(f-em))$. Also,
\begin{align*}
Pz&= aPv^{\perp}+bPv\\
&=(ag,ah)+(br,0)\\
&=(ag+br,ah).
\end{align*}

\vspace{0.05in}

So in order for $\langle x,Pz\rangle \langle y,Pz\rangle =c$ we must have for some $x$: 

\vspace{0.05in}

\[ ag+\sqrt{1-a^2}r=x\mbox{ and }ah=mx+\frac{c}{x}(f-em).\]

\vspace{0.1in}

Substituting the first equation into the second equation for $x$, since all constants except $a$ are fixed,
we have a quadratic equation in $a$.
So A is of cardinality at most two and the lines through A are a set of measure zero in $W_1^{\perp}$.

\vspace{0.15in}

\noindent {\bf Case 2:} $PW_1^{\perp}= \emptyset$.
\vskip10pt
In this case, if $z\in A$ then $Pz=0$. So $\langle z,x_j\rangle = \langle Pz,x_j\rangle =0$. By (6), this is a contradiction unless
$A=\emptyset$.

\vspace{0.15in}

\noindent {\bf Case 3:} $PW_1^{\perp}$ is dimension one in $X$.

\vspace{0.15in}

In this case,
$x_j^{\perp}\cap W_1^{\perp}=y_j^{\perp}\cap W_1^{\perp}$. So $x_j^{\perp}\cap W_1^{\perp}$ is a 
hyperplane in $W_1^{\perp}$.  
Let $y_j=(e,f)\in \RR^2$ and choose $\|v\|=1$ in $W_1^{\perp}$ with $Pv^{\perp}=0$. Let $Pv=(g,h)$.
\\

Then, if $z\in A$ with $z=av+bv^{\perp}$
 by (8) and (*),

\[Pz=(ag,ah)=(x,mx+\frac{c}{x}(f-em))\] 
Hence, \[ 
ag=x\mbox{ and }ah=mx+\frac{c}{x}(f-em).\]
Substituting the first equation in the second equation for $x$, all constants except $a$ are fixed so this is a quadratic equation for $a$
and hence has at most two solutions. That is, $A$ consists of at most two lines and so is a set of measure zero in $W_1^{\perp}$.
 
\vskip10pt
(10) We check the case $dim\ W_1^{\perp}\ge 3$.
\vskip10pt
Since $x_j^{\perp}\cap y_j^{\perp}$ is a subset of $\RR^n$ of codimension 1 or 2, we will examine these two cases.
\vskip10pt
\noindent {\bf Case 1:} $(x_j^{\perp}\cap y_j^{\perp})\cap W_1^{\perp}$ is a set of codimension 2 in $W_1^{\perp}$.
\vskip10pt
Choose a two dimensional subspace $V$ of $W_1^{\perp}$ so that $W_1^{\perp}=V\oplus (x_j^{\perp}\cap y_j^{\perp})$.

\vspace{0.05in}

{\bf Claim}: $PV=X$.

\vspace{0.1in}

{\bf Proof of claim: }If not, then there is a $v\in V$ with $Pv=0$. That is, 
\[ \langle v,x_j\rangle =\langle Pv,x_j\rangle =0= \langle v,y_j\rangle.\]
{\hspace{0.2in}}But $v\perp (x_j^{\perp}\cap y_j^{\perp}$), a contradiction.

\vspace{0.15in}

Choose $v\in V$, $\|v\|=1$ and $Pv=re_1$. Let $v^{\perp}$ be in V and orthogonal to $v$ and
$Pv^{\perp}=(g,h)$.
 It follows that every $z\in W_1^{\perp}$
is of the form $$z=bv+av^{\perp}+u$$ when $a^2+b^2+\|u\|^2=1$, and $P(u)=0$. 
By (8) and (*), $$Pz=(x,mx+\frac{c}{x}(f-em))$$ But  
\begin{align*}
Pz&= bPv+aPv^{\perp}\\
&=(br,0)+(ag,ah)\\
&=(ag+br,ah).
\end{align*}
So we must have for some $x$:
\[ ag+br=x\mbox{ and }ah=mx+\frac{c}{x}(f-em).\]
Now, from the second equation
\[ ah=m(ag+bj)+\frac{c}{ag+br}(f-em),\]
All the constants here are fixed except a,b. This equation, represents a quadratic equation in a,b. Hence,

\[\{z=av^{\perp}+bv:\langle x_j,Pz\rangle \langle y_j,Pz\rangle =c\}\] is the union of two continuous
quadratic graphs $G_1,G_2$ in $V$. 
For $z\in A$, $z=av+bv^{\perp}+u$, $\|z\|=1$.
 So
\[ A=\{x+u:x\in G_1,\  \|u\|^2=1-\|x\|^2\} \]
or
\[A=\{x+u:x\in G_2,\  \|u\|^2=1-\|x\|^2\}.\]

\vspace{0.2in}

In each case, since the unit ball of $(x_j^{\perp}\cap y_j^{\perp})\cap W_1^{\perp}$ is a set of measure zero in $W_1^{\perp}$ and
A is a curve on the sphere of $W_1^{\perp}$  and since $dim\ W_1^{\perp}\ge 3$, the
lines through A are a set of measure zero in $W_1^{\perp}$.

\vspace{0.1in}

\noindent {\bf Case 2:} $(x_j^{\perp}\cap y_j^{\perp})\cap W_1^{\perp}$ is a set of codimension 1 in $W_1^{\perp}$.
\vskip10pt
Choose $v\in W_1^{\perp}$ with $v\perp x_j$. Then every $z\in W_1^{\perp}$, $\|z\|=1$, is of the form
$z=av+u$ and $u\in x_j^{\perp}$. Let $Pv=(g,h)$. Now, $Pz=Pav$ so 
$$Pz=aPv=(ag,ah).$$ Hence, by (8) and (*), 
\[ ag = x\mbox{ and }ah=mx+\frac{c}{x}(f-em).\]
This yields
\[ ah=mag+\frac{c}{ag}(f-em).\]

 \vspace{0.1in}
 
This quadratic equation in $a$ has at most two solutions $a_1,a_2$. So $A$ is of the form $$\{a_iv+u:\|u\|^2=1-a_i^2\}.$$
 
 \vspace{0.1in}

But these sets are two translates of a circle in $(x_j^{\perp}\cap y_j^{\perp})\cap W_1^{\perp}$ and the lines through
these sets form a set of measure zero in $W_1^{\perp}$.

\vspace{0.15in}

(11) Now, for any $x_j$ the set of all $z\in W_1^{\perp}$ with $P_zx_j\in y_1^{\perp}$ is a set of measure zero in $W_1^{\perp}$.
So choose $z\in W_1^{\perp}$ which is in none of these $k$ subsets of $W_1^{\perp}$ of measure zero. Now, for any $0\not= x\in \RR^n$,
if $x\not= x_j$ for any $j \in [k]$, then $\{P_ix\}_{i=2}^m$ spans $\RR^n$. And if $x=x_j$ for some $j\in [k]$ then
$P_zx_j\notin y_j^{\perp}$ and so $\{P_zx_j,P_ix_j\}_{i=2}^m$ spans $\RR^n$. That is, $\{P_z,P_i\}_{i=2}^m$ does phase
retrieval in $\RR^n$.

\vspace{0.15in}

This completes the proof of the theorem.
\end{proof}

\vspace{0.1 in}

As a consequence, applying the theorem over and over, we have
\vspace{0.1 in}

\begin{corollary} Assume a family of subspaces $\{W_i\}_{i=1}^m$ with dim $W_i=k_i$, $1\le k_i\le n-1$, for all $i=1,2,\ldots,$ does phase retrieval in $\RR^n$.
The following hold:
\begin{enumerate}
\item For any $1\le k_i\le r_i \le n-1$, there are subspaces $W_i'$ of $\RR^n$ so that $W_i\subset W_i'$,
dim $W_i'=r_i$ and $\{W_i'\}_{i=1}^m$
does phase retrieval.
\item 
 There are hyperplanes $\{W_i'\}_{i=1}^m$ with $W_i\subset W_i'$
and $\{W_i'\}_{i=1}^m$ does phase retrieval. So by Theorem \ref{T}, $m\ge 2n-2$.
\end{enumerate}
\end{corollary}

\end{document}